\documentclass[11pt]{amsart}

 \usepackage[T1]{fontenc}
 \usepackage[utf8]{inputenc}
 \usepackage[english]{babel}
 \usepackage{graphicx}
 \usepackage{geometry}

 \usepackage{amsmath,amsthm,amsfonts,amssymb}
 \usepackage{enumerate}
 \usepackage{color}
 \usepackage{fancyhdr}
 \usepackage[titletoc]{appendix}
 \usepackage{hyperref}
 \usepackage{tikz}
 \usepackage{tikz-cd}
 \usepackage{caption}
 \usetikzlibrary{matrix,arrows,decorations.pathmorphing, positioning}

\newtheorem{theo}{Theorem}[section]
\newtheorem{coro}[theo]{Corollary}
\newtheorem{lemma}[theo]{Lemma}

\theoremstyle{definition}
\newtheorem{defi}[theo]{Definition}
\theoremstyle{definition}
\newtheorem{rem}[theo]{Remark}
\newtheorem{example}[theo]{Example}

\input{xy}
\xyoption{all}
\xyoption{poly}
\usepackage[all]{xy}

 \newcommand\R{{\mathbb{R}}}
 
 \newcommand\Z{{\mathbb{Z}}}

 \newcommand\T{{\mathbb{T}}}

 \DeclareMathOperator{\sys}{sys}
 \DeclareMathOperator{\Area}{Area}
 \DeclareMathOperator{\Vol}{Vol}

 \DeclareMathOperator{\length}{length}

\title[A systolic inequality for 2-complexes]{A systolic inequality for 2-complexes of maximal cup-length and systolic area of groups}
   \author{Eugenio Borghini}
   \address{Departamento  de Matem\'atica - IMAS\\
 FCEyN, Universidad de Buenos Aires. Buenos Aires, Argentina.}
\email{eborghini@dm.uba.ar}

\thanks{Researcher of CONICET. Partially supported by grant UBACyT 20020160100081BA}

\subjclass[2010]{53C23, 57M20, 57R95, 57N16}

\keywords{Systolic area, surface groups.}

\begin{document}

   % abstract 
   \begin{abstract}
			We extend a systolic inequality of Guth for Riemannian manifolds of maximal $\Z_2$ cup-length to piecewise Riemannian complexes of dimension 2. As a consequence we improve the previous best universal lower bound for the systolic area of groups for a large class of groups, including free abelian and surface groups, most of irreducible 3-manifold groups, non-free Artin groups and Coxeter groups (or more generally), groups containing an element of order 2.
   \end{abstract}

   \maketitle

\section{Introduction}

The systole $\sys(X)$ of a metric space $X$ is a metric invariant of $X$ which consists of the length of a shortest non-contractible loop in $X$. In 1949, Loewner proved that $\Area(\T^{2},g) \geq \frac{\sqrt{3}}{2} \sys(\T^{2},g)^2$ for any Riemannian metric $g$ on the 2-torus $\T^{2}$, a result that can be viewed as a kind of reverse isoperimetric inequality. His student Pu \cite{Pu} extended the inequality to the projective plane and later, Accola \cite{Ac} and Blatter \cite{Bl} independently proved a similar inequality for general orientable Riemannian surfaces. The systematic study of systolic phenomena and invariants was initiated by Gromov in 1983 \cite{Gr1}, who proved the existence of a universal constant $C_n$, depending only on the dimension $n$, such that $\Vol(X,g) \geq C_n \sys(X,g)^{n}$ for all $n$-dimensional \textit{essential} polyhedra $X$ equipped with a piecewise Riemannian metric $g$. Using minimal hypersurfaces, Guth \cite{Gu} established a systolic inequality with a better constant $C_n$ for closed Riemannian manifolds whose $\Z_2$-cohomology has maximal cup-length. This result was subsequently refined and generalized in various directions (see for example \cite{BK, BKP, N}). The main result of the present article is an extension of this systolic inequality in dimension $2$ to complexes of maximal $\Z_2$ cup-length.
\begin{theo}\label{main}
Let $(X,g)$ be a compact connected simplicial complex of dimension $2$ equipped with a piecewise Riemannian metric $g$. Suppose there exist classes $\alpha$, $\beta$ in $H^{1}(X, \Z_2)$ such that $\alpha \cup \beta \neq 0$ in $H^{2}(X, \Z_2)$. Then, $\Area(X, g) \geq \frac{1}{2} \sys(X,g)^{2}$.
\end{theo}
One of our main motivations to study the systolic geometry of $2$-dimensional complexes is its connection with finitely presentable groups. Define the \textit{systolic area} $\sigma(X)$ of a $2$-dimensional complex $X$ as
\[
	\sigma(X) := \inf_g \frac{\Area(X,g)}{\sys(X,g)^2},
\]
the infimum being over all piecewise Riemannian metrics $g$ on $X$. Gromov introduced in \cite{Gr2} the notion of systolic area of a finitely presentable group $G$ as 
\[
	\sigma(G) := \inf_{X} \sigma(X),
\]
where the infimum is taken over the finite 2-dimensional simplicial complexes $X$ with fundamental group isomorphic to $G$. We are interested in understanding how the systolic area of a closed surface $S$ compares to the systolic area of its fundamental group, beyond the obvious inequality $\sigma(\pi_1(S)) \leq \sigma(S)$. A connection between these invariants was revealed in \cite{BPS}, where the authors showed that the systolic area of fundamental groups of surfaces grows asymptotically with the genus $g$ as $\frac{g}{\log(g)^{2}}$ and thus exhibits the same asymptotic behavior as the systolic area of surfaces. Hence, the most optimistic conjecture would be that $\sigma(\pi_1(S)) = \sigma(S)$. If this were true, it would imply that in some sense the most efficient geometric-topological model of a surface group is the surface itself. In \cite{BM} we showed that this is indeed the case for a combinatorial approximation of the systolic area called \textit{simplicial complexity}, which was recently introduced by Babenko, Balacheff and Bulteau \cite{BBB}. Our focus here is on improving the universal lower bound for the systolic area of surface groups (that is, fundamental groups of closed surfaces). Let us briefly review the known estimates for the systolic area of groups before stating our result. Gromov showed that the systolic area of non-free groups is strictly positive: it follows from \cite[Theorem 6.7.A]{Gr1} that $\sigma(G) \geq \frac{1}{10^{4}}$ for non-free groups $G$. This lower bound was considerably improved by the works of Katz, Rudyak and Sabourau \cite{KRS}, Rudyak and Sabourau \cite{RS}, and Katz, Katz, Sabourau, Shnider and Weinberger \cite{KKSSW}, and now it is known that $\sigma(G) \geq \frac{1}{4}$ for non-free groups. In contrast, stronger bounds are available for the systolic area of surfaces. By an inequality of Gromov (see \cite[\S 5.2]{Gr1}), the systolic area of aspherical closed surfaces is at least $\frac{3}{4}$. Combined with Pu's optimal systolic inequality for the projective plane, this leads to the estimate $\sigma(S) \geq \frac{2}{\pi}$ for every non-simply connected closed surface $S$. One might ask if the constant $\frac{2}{\pi}$ is optimal also for the systolic area of surface groups (or even for non-free groups, cf. \cite[Question 1.5]{RS}). As a consequence to our main result we refine the universal lower bound for the systolic area of a large class of groups which includes the surface groups, bringing it closer to $\frac{2}{\pi}$. More precisely, we show that for a group in this class, which we call \textit{surface-like} groups and is defined by a cohomological condition (see Definition \ref{surflike} below), the systolic area is bounded from below by $\frac{1}{2}$. The use of a covering theory argument borrowed from \cite{KKSSW} allows us to establish the following more general result.
\begin{theo}\label{groupsysarea}
Let $G$ be group which contains a surface-like subgroup $T$. Then, $\sigma(G) \geq \frac{1}{2}$.
\end{theo}
Notably, (infinite) fundamental groups of closed irreducible 3-manifolds satisfy this hypothesis by the positive solution of the surface subgroup conjecture by Kahn and Markovic \cite{KM}.
\par Lastly, we remark that if we assume a strengthened version of Gromov's Filling Area Conjecture the inequality from Theorem \ref{groupsysarea} can be promoted to $\sigma(\pi_1(S)) \geq \frac{2}{\pi}$ for \textit{non-orientable} closed surfaces $S$. This follows from adapting an argument of Ivanov and Katz (\cite[Proposition 3.1]{IK}) to our context.

\par \textbf{Acknowledgments}. The author is glad to thank Gabriel Minian for his help and interest during the preparation of this article, Stéphane Sabourau for his helpful comments (especially with regard to Corollary \ref{cover}) and Marcos Cossarini for pointing out an imprecision in a previous version of Remark \ref{FAC}.

\section{Systolic inequality for maximal cup-length 2-complexes}\label{main_section}

Throughout the article, the complexes will be finite and connected unless otherwise stated and the groups will always be finitely presentable. By the cohomology ring of a group $G$ we will understand its cohomology as a discrete group, i.e. the cohomology of an Eilenberg-MacLane space $K(G, 1)$. We will work with reduced (co)homology and the coefficient ring for (co)homology groups will be $\Z_2$. This section is devoted to the proof of our main Theorem \ref{main}. We outline here the argument of the proof, which consists mainly of two parts. First, we realize a conveniently chosen $\Z_2$-homology class of a given $2$-complex $X$ with maximal $\Z_2$ cup-length by a continuous map $f:S \to X$ from a closed surface $S$. Then we derive a lower bound for the systolic area of $S$ by applying a result of Nakamura which refines Guth's systolic inequality (see \cite[Theorem 2.2]{N} and Theorem \ref{Nakasurface} below), and pushforward this inequality via $f$.
\par To carry out the first part of the argument, we need to realize 2-dimensional homology classes as continuous images of surfaces in a controlled way. As it is well-known, Thom showed that a $\Z_2$-homology class of any dimension can be represented as the image of the fundamental class of a closed manifold. We record here an explicit construction for the case of 2-dimensional $\Z_2$-homology classes of simplicial complexes similar to the one described in (\cite[pp. 108-109]{Ha}). Recall that a simplicial map is \textit{non-degenerate} if it preserves the dimension of the simplices.

\begin{lemma}\label{homology}
Let $X$ be a simplicial complex and let $[C] \in H_2(X)$ be a non-trivial homology class. Then, there exists a triangulated closed surface $S$ (possibly non-orientable and not connected) together with a non-degenerate simplicial map $h:S \to X$ such that $h_{*}[S] = [C]$, where $[S] \in H_2(S)$ denotes the fundamental class. Moreover, $h$ does not identify 2-simplices, meaning that $h(\sigma) \neq h(\eta)$ for two distinct 2-simplices $\sigma$, $\eta$ of $S$.
\end{lemma}

Notice in particular that for a piecewise Riemannian 2-complex $(X,g)$, the map $h:S \to X$ from the statement preserves lengths and areas if the surface $S$ is endowed with the piecewise Riemannian pullback metric $h^{*}(g)$.

\begin{proof}
Take a 2-cycle $Z = \sum_{i} \sigma_i$ in $C_2(X)$ representing the homology class $[C]$ and form a disjoint union of 2-simplices $\tilde{\sigma}_i$, one for each $\sigma_i$ in the support of $Z$. Since the algebraic boundary of $Z$ is trivial, the edges of the simplices $\sigma_i$ cancel in pairs. Choose a maximal set of such canceling pairs and identify the edges of $\tilde{\sigma_i}$ accordingly. It is clear that the quotient space obtained from $\coprod_{i} \tilde{\sigma_i}$ by performing these identifications is a closed surface $S$ and that it gives rise to a simplicial map $h:S \to X$ with the desired properties.
\end{proof}

\begin{rem}
In general, for a $\Z_2$-homology class of dimension $n \geq 2$ the construction from the proof gives a realization by a pseudo-manifold whose singularities are of codimension at least $3$ (cf. \cite[p. 109]{Ha}).
\end{rem}

We reproduce next the precise result from \cite{N} that we will need. Notice that the original result applies to closed manifolds of any dimension, but we state it only for surfaces which is our case of interest. Recall that given a piecewise Riemannian complex $(X,g)$ and a non-trivial homology class $\gamma \in H_{1}(X)$, we denote by $\length(\gamma)$ the infimum of the lengths of 1-cycles representing $\gamma$. Following Guth (see \cite{Gu}), for a non-trivial cohomology class $\alpha$ we define its length as
\[
	\length(\alpha) := \inf \{ \length(\gamma): \gamma \in H_1(X), \, \alpha(\gamma) \neq 0\}.
\]
Observe that $\length(\alpha) \geq \sys(X,g)$ for any non-trivial $\alpha \in H^{1}(X)$.

\begin{theo}\label{Nakasurface}(cf. \cite[Theorem 2.2]{N})
Let $(S, g)$ be a closed Riemannian surface. Let $\alpha$, $\beta \in H^{1}(S)$ be not necessarily distinct classes such that $\alpha \cup \beta \neq 0$ in $H^{2}(S)$ and let $2R := \min \{\length(\alpha), \length(\beta)\} > 0$. Then, there exists $x \in S$ such that for any $r \in (0, R)$,
\[
	\Area B(x, r) \geq \frac{(2r)^{2}}{2}.
\]
\end{theo}

Keeping the notations from the previous statement, since $2R \geq \sys(S,g)$ it follows that $\Area(S,g)$ is greater than or equal to $\frac{\sys(S,g)^{2}}{2}$. Thus, this result does not improve the best known systolic inequalities for surfaces. Rather, its power lies in the fact that it provides a lower bound for the area of a surface only from the existence of two ``long'' cohomology classes with non-trivial cup product regardless of the size of the systole.

\begin{proof}[Proof of Theorem~\ref{main}]
Since by hypothesis there are classes $\alpha$, $\beta \in H^{1}(X)$ such that $\alpha \cup \beta \neq 0$, there exists a homology class $[C] \in H_{2}(X)$ with $\alpha \cup \beta [C] \neq 0$. By applying Lemma \ref{homology} to the class $[C]$, we obtain a triangulated closed surface $S$ together with a simplicial map $h:S \to X$ representing $[C]$. Endow $S$ with the pullback metric $h^{*}(g)$, where $g$ is the piecewise Riemannian metric on $X$. By the naturality of the cup product, we have
\[
	h^{*}(\alpha) \cup h^{*}(\beta) [S] = (\alpha \cup \beta) h_{*}[S] = \alpha \cup \beta [C] \neq 0.
\]
If $S$ is not connected, the previous computation implies that for some connected component of $S$ the corresponding components of $h^{*}(\alpha)$ and $h^{*}(\beta)$ have non-trivial cup product. By a slight abuse of notation, we will still call $S$ such a component. Now both the length of $h^{*}(\alpha)$ and $h^{*}(\beta)$ are at least $\sys(X,g)$. Indeed, if $\gamma$ is a 1-cycle in $S$ such that $h^{*}(\alpha) \gamma \neq 0$, since $h^{*}(\alpha) \gamma = \alpha (h_{*}\gamma)$ and $h$ is length preserving it follows that $\length(h^{*}(\alpha)) \geq \length(\alpha) \geq \sys(X,g)$. Analogously, $\length(h^{*}(\beta)) \geq \sys(X,g)$. It is not difficult to show that Theorem \ref{Nakasurface} applies also for surfaces with a piecewise smooth Riemannian metric, since one may approximate up to an arbitrarily small error such a metric by a smooth one. Hence, by Theorem \ref{Nakasurface} there is a point $x$ in $S$ such that for every $r \in \left(0, \frac{\sys(X, g)}{2}\right)$
\[
	\Area B_{S}(x, r) \geq \frac{(2r)^{2}}{2},
\]
where $B_S(x, r)$ stands for the ball of radius $r$ centered at the point $x \in S$. Since $h$ preserves areas we have $\Area(S, h^{*}(g)) \leq \Area(X,g)$, which implies that $\Area(X, g) \geq \frac{1}{2} \sys(X,g)^{2}$ as desired.
\end{proof}

In fact, notice that it follows from the argument in the proof that there exists a point $y \in X$ (namely, $h(x)$) such that $\Area B_{X}(y,r) \geq \frac{(2r)^{2}}{2}$ for every $r \in \left(0, \frac{\sys(X, g)}{2}\right)$. Moreover, the same conclusion can be obtained also for the area of a ball in non-compact 2-complexes, with the caveat that the systole of a non-compact space might equal zero. This local reformulation of Theorem \ref{main} is especially relevant when dealing with covering spaces. As it was observed in \cite{KKSSW}, covering projections of piecewise Riemannian complexes map injectively balls of radius less than half the systole of the base space. Since on the other hand the systole of a non-simply connected covering space is clearly greater than or equal to the systole of the base space, we obtain the following corollary.

\begin{coro}\label{cover}
Let $(X,g)$ be a compact connected complex of dimension $2$ equipped with a piecewise Riemannian metric $g$ and let $(\hat{X},\hat{g})$ be a covering of $X$. Suppose that there exist classes $\alpha$, $\beta$ in $H^{1}(\hat{X}, \Z_2)$ such that $\alpha \cup \beta \neq 0$ in $H^{2}(\hat{X}, \Z_2)$ and let $2R := \min\{\length(\alpha), \length(\beta)\}$. Then, there exists $x \in X$ such that $\Area B(x, r) \geq \frac{(2r)^{2}}{2}$ for all $r \in (0, R)$. In particular $\Area(X,g) \geq \frac{1}{2}\sys(X,g)^{2}$.
\end{coro}

\section{Systolic area of groups}\label{groups}

In this section we derive an inequality for the systolic area of a class of groups which contains the surface groups. We start by introducing the notion of \textit{surface-like} groups.

\begin{defi}\label{surflike}
Let $G$ be a group. We say that $G$ is \textit{surface-like} if there exist classes $\alpha$, $\beta$ in $H^{1}(G)$ such that $\alpha \cup \beta \neq 0$ in $H^{2}(G)$.
\end{defi}

\begin{example}\label{surflike_ex}
As the name suggests, all surface groups are surface-like by Poincaré Duality. But many other groups are surface-like: free abelian groups of rank $\geq 2$, elementary abelian $2$-groups, direct and free products of a group with a surface-like group are some other examples.
\end{example}

We are now ready to prove Theorem \ref{groupsysarea}.

\begin{proof}[Proof of Theorem~\ref{groupsysarea}]
Let $(X,g)$ be a piecewise Riemannian 2-complex with fundamental group $G$ and let $(\hat{X},\hat{g})$ be the covering of $X$ with fundamental group $T$ endowed with the pullback metric $\hat{g}$. By Corollary \ref{cover}, the proof reduces to verifying that under the hypothesis on its fundamental group, $\hat{X}$ must have maximal $\Z_2$ cup-length. Notice that the complex $\hat{X}$ includes as the 2-skeleton of an aspherical CW-complex $K$ (possibly infinite dimensional). Thus $K$ is an Eilenberg-MacLane space $K(T,1)$ and so its cohomology ring $H^{*}(K)$ is isomorphic to the cohomology ring $H^{*}(T)$ of $T$. By construction, the inclusion $\hat{X} \hookrightarrow K$ induces an isomorphism $H^{1}(K) = H^{1}(T) \to H^{1}(\hat{X})$ and a monomorphism $H^2(K) = H^{2}(T) \to H^{2}(\hat{X})$. Since $T$ is surface-like, this implies that $\hat{X}$ has maximal $\Z_2$ cup-length.
\end{proof}

Besides the groups listed in Example \ref{surflike_ex}, the bound from Theorem \ref{groupsysarea} applies among others to non-free Artin groups, groups containing an element of order 2 (in particular, to Coxeter groups) and infinite fundamental groups of closed irreducible 3-manifolds. Indeed, Artin groups contain a copy of $\Z \oplus \Z$ unless they are free, while groups with an element of order 2 contain a copy of $\Z_2$. As for the irreducible 3-manifold groups, recall that the surface subgroup conjecture states that every closed, irreducible 3-manifold with infinite fundamental group contains an immersed $\pi_1$-injective closed surface. This conjecture was settled in the affirmative by Kahn and Markovic in \cite{KM} and thus, by Theorem \ref{groupsysarea} the systolic area of infinite fundamental groups of closed irreducible 3-manifolds also admits $\frac{1}{2}$ as a lower bound.

\begin{rem}\label{FAC}
As it was announced in the introduction, the inequality of Theorem \ref{groupsysarea} can be improved for fundamental groups of non-orientable surfaces if we assume a strengthened version of Gromov's Filling Area Conjecture. A Riemannian manifold $(M^{n+1}, g_M)$ is an \textit{isometric filling} of a Riemannian manifold $(N^{n},g_N)$ if $\partial M = N$ and the restriction to $N \times N$ of the distance function $d_{g_M}$ determined by $g_M$ coincides with $d_{g_N}$. The Filling Area Conjecture states that the minimum area among orientable isometric fillings of the circle equipped with its standard Riemannian metric of length $2\ell$ is attained at the standard hemisphere of area $\frac{2}{\pi} \ell^{2}$.
\par Let $(X,g)$ be a piecewise Riemannian 2-complex with fundamental group isomorphic to the fundamental group of a non-orientable surface and suppose that $h:S \to X$ represents the relevant homology class of $H_2(X)$ with $S$ a connected surface. By the naturality of the exact sequence of the universal coefficient theorem, $S$ must be non-orientable. Now we proceed as in the proof of \cite[Proposition 3.1]{IK}. Take the shortest loop $\gamma \in S$ such that $h_{*}[\gamma]$ is non-trivial in $\pi_1(X)$ and open up $S$ along $\gamma$. The resulting Riemannian surface $(\Sigma, \tilde{g})$ is an isometric filling (possibly non-orientable) of a circle of twice the length of $\gamma$ and of the same area as $S$. If we assume that the conclusion of Filling Area Conjecture holds (also for non-orientable isometric fillings), we have $\Area(\Sigma, \tilde{g}) \geq \frac{2}{\pi} (\length(\gamma))^{2} \geq \frac{2}{\pi}(\sys(X,g))^{2}$. Hence, since $\Area(X,g) \geq \Area(\Sigma,\tilde{g})$ we conclude that $\sigma(X, g) \geq \frac{2}{\pi}$. In particular, this would imply that $\sigma(\Z_2) = \frac{2}{\pi}$.
\end{rem}

\end{document}